\documentclass[11pt]{amsart}
\usepackage[utf8]{inputenc}

\usepackage[margin=3.5cm]{geometry}
\usepackage{lipsum}

\usepackage{float}

\usepackage{enumitem}
\usepackage{amsfonts}
\usepackage{amsmath}
\usepackage{amsthm}
\usepackage{amssymb}
\usepackage{mathrsfs}
\usepackage{enumitem}
\usepackage{forest}

\usepackage[all]{xy}
\usepackage{appendix}

\usepackage{tikz}
\usetikzlibrary{shapes.geometric}
\usetikzlibrary{positioning}
\usetikzlibrary{automata}
\usetikzlibrary{decorations}
\usetikzlibrary{decorations.pathreplacing}

\usepackage{tikz-cd}

\theoremstyle{definition}
\newtheorem{defi}{Definition}[section]

\newtheorem{prop}[defi]{Proposition}
\newtheorem{cor}[defi]{Corollary}

\newtheorem{theo}[defi]{Theorem}

\definecolor{indiagreen}{rgb}{0.07, 0.53, 0.03}

\makeatletter
\newcommand\RSloop{\@ifnextchar\bgroup\RSloopa\RSloopb}
\makeatother
\newcommand\RSloopa[1]{\bgroup\RSloop#1\relax\egroup\RSloop}
\newcommand\RSloopb[1]%
  {\ifx\relax#1%
   \else
     \ifcsname RS:#1\endcsname
       \csname RS:#1\endcsname
     \else
       \GenericError{(RS)}{RS Error: operator #1 undefined}{}{}%
     \fi
   \expandafter\RSloop
   \fi
  }
\newcommand\X{0}
\newcommand\RS[1]%
  {\begin{tikzpicture}
     [every node/.style=
       {circle,draw,fill,minimum size=1.5pt,inner sep=0pt,outer sep=0pt},
      line cap=round
     ]
   \coordinate(\X) at (0,0);
   \RSloop{#1}\relax
   \end{tikzpicture}
  }
\newcommand\RSdef[1]{\expandafter\def\csname RS:#1\endcsname}
\newlength\RSu
\RSdef{A}{\draw[indiagreen, thick] (\X) -- +(135:\RSu) coordinate(\X L);\edef\X{\X L}}
\RSdef{a}{\draw[indiagreen, thick] (\X) -- +(135:\RSu) node{};}

\RSdef{B}{\draw[indiagreen, thick] (\X) -- +(90:\RSu) coordinate(\X I);\edef\X{\X I}}
\RSdef{b}{\draw[indiagreen, thick]  (\X) -- +(90:\RSu) node{};}

\RSdef{C}{\draw[indiagreen, thick] (\X) -- +(45:\RSu) coordinate(\X R);\edef\X{\X R}}
\RSdef{c}{\draw[indiagreen, thick] (\X) -- +(45:\RSu) node{};}

\RSdef{L}{\draw (\X) -- +(135:\RSu) coordinate(\X L);\edef\X{\X L}}
\RSdef{l}{\draw (\X) -- +(135:\RSu) node{};}

\RSdef{I}{\draw (\X) -- +(90:\RSu) coordinate(\X I);\edef\X{\X I}}
\RSdef{i}{\draw (\X) -- +(90:\RSu) node{};}

\RSdef{R}{\draw (\X) -- +(45:\RSu) coordinate(\X R);\edef\X{\X R}}
\RSdef{r}{\draw (\X) -- +(45:\RSu) node{};}

\title{Asymptotic lengths of permutahedra and associahedra}

\author{Daria Poliakova}
\dedicatory{To Jim Stasheff for his birthday with gratitude}

\begin{document}
\maketitle
 \begin{abstract}
We define asymptotic lengths for families of oriented polytopes. We show that permutahedra with weak order orientations have asymptotic total length 1 and associahedra with Tamari order orientations have asymptotic total length 1/2.
\end{abstract}

\section*{Introduction}

Explicit cellular approximations of polytopal diagonals are an area with questions in combinatorics and applications in algebra. Cellular diagonals for simplices/cubes are folklore (or Alexander-Whitney/ Serre maps respectively) and are used to define cup products in singular/singular cubic cohomology. Cellular diagonals for associahedra were combinatorially described in \cite{SU04} and revisited in \cite{Lod11}, \cite{MS}, \cite{MTTV}, \cite{BCP}; algebraically, they allow for tensoring $A_\infty$-algebras. Cellular diagonals for permutahedra were combinatorially described in \cite{SU04}, with an exciting topological reinterpretation in \cite{LA}--\cite{DOLAPS}, producing formulas for all the generalized permutahedra in the sense of \cite{Pos}, including associahedra, freehedra \cite{San}--\cite{RS}--\cite{PP}, operahedra \cite{LA}, multiplihedra \cite{For}, \cite{LM}...

In all the cases above, the combinatorial answers involve orientations on 1-skeleta of polytopes: tautological for simplices and cubes, Tamari order for associahedra, weak Bruhat order for permutahedra. For simplices, cubes and associahedra the images of the cellular diagonals are easy to describe in terms of these orientations: the {\em magic formula} says that a pair of faces $(F,F')$ is in the image of the cellular diagonal if and only if the top vertex of $F$ is less or equal to the bottom vertex of $F'$ in the respective partial order. For permutahedra, this is no longer true. In the terminology to be introduced in the current note, this is a manifestation of the fact that simplices, cubes and associahedra are $2$-short, while permutahedra are not.

Taking one step further, one may ask if aforementioned diagonals are coassociative. The answer is positive for simplices and cubes, and negative for associahedra and permutahedra, where one may therefore look for explicit higher coherences correcting for non-coassociativity of the respective maps. Are these coherences given by higher magic formulas? In the terminology of \cite{AP} and the current note, the answer conjecturally depends on whether the polytope in question is short. And while short polytopes with non-coassociative diagonals exist, neither associahedra nor permutahedra are short.

The current note is therefore dedicated to measuring the lack of shortness. For families of polytopes,  the notions of asymptotic $k$-lengths and asymptotic total lengths are introduced, varying between 0 and 1 and being 0 for short polytopes. Associahedra are proved to have asymptotic total length 1/2. Permutahedra are proved to have asymptotic total length 1.

Many mathematicians were at some point of their life approached by fermatists -- yes, also those working far from number theory, and yes, also after Fermat's Last Theorem has {\em really} been proved. The fantasies of my first fermatist involved coding, and he proudly informed me, that his code had proved the theorem by 99 \%. Constructing explicit higher coherences for the non-coassociative diagonals of both associahedra and permutahedra are beautiful open problems, towards which this note makes absolutely no progress. One can nevertheless say, in the spirit of my first fermatist, that this paper demonstrates: the solution for associahedra will be by 50 \% easier than for permutahedra.

{\bf Acknowledgements.} 
I am thankful to the anonymous referee of \cite{Pol} who asked about details of non-shortness for associahedra and permutahedra. The work was funded by the Deutsche Forschungsgemeinschaft (DFG, German Research Foundation) – SFB-Geschäftszeichen 1624 – Projektnummer 506632645.

\section{Asymptotic lengths}

Let $P \subset \mathbb{R}^n$ be a convex polytope and $\psi$ a linear functional not constant along any edge of $P$. Then $\psi$ equips the 1-skeleton of $P$ with an orientation and thus makes the vertex set of $P$ into a poset. We extend this relation to all faces of $P$ by saying that $F \leq  G$ if the top vertex of $F$ is less or equal to the bottom vertex of $G$ with respect to the partial order on vertices (note that the extended relation is not reflexive: $F \leq F$ is only true when $F$ is a vertex). In the rest of the paper, a polytope always means a polytope with a fixed orientation on its 1-skeleton, but we often omit the latter for conciseness.

In the above context we can look at {\em face chains} $F_1 \leq F_2 \leq \ldots \leq F_k$ within $P$. The following definition was introduced in \cite{AP}.

\begin{defi}
The excess of the face chain $\mathbf{F} = (F_1 \leq F_2 \leq \ldots \leq F_k)$ within $P$ is 
$$ {\mathrm{exc}}(\mathbf{F}) = (\dim P-1) - \sum_{i=1}^k (\dim F_i -1).$$
\end{defi}

If $k=1$, this is just the definition of codimension, and nontrivial faces certainly have positive codimension. For $k>1$ this is no longer true in general: excesses can get zero or negative. This gives rise to the definition of shortness, as in \cite{AP}.

\begin{defi}
    An polytope $P$ is short, if within each of its faces excesses of nontrivial chains are positive.
\end{defi}

Families of short polytopes include simplices, cubes, and freehedra of \cite{San} and \cite{RS} in the convex realization of \cite{PP}, as demonstrated in \cite{Pol}. Short polytopes are best equipped for the study of polytopal diagonals, whose images are then described by the magic formulas, i.e. consist of face chains of length 2.

If a polytope is not short, one can ask: how negative can excesses get? Clearly the theoretical lower bound is the excess of a chain consisting of facets: for a length $k$ chain in a dimension $n$ polytope this is $$E_k(n) = (n-1)-k(n-2) = (1-k)n+2k-1.$$
For a polytope $P$ of dimension $n$, denote by $e_k(P)$ the minimum excess of a length $k$ chain in $P$. A single polytope $P$ certainly cannot have $e_k(P) = E_k(n)$ for all $k$, simply because the number of facets is finite. However, we can look at {\em families} of polytopes $P^\bullet$ with a polytope $P^n$ for each dimension $n$, and introduce the following definition:

\begin{defi}
    A family of polytopes $P^\bullet$ is
    \begin{itemize}
        \item $k$-long, if there exists a value $N(k)$ such that for $n>N(k)$, $e_k(P^n) = E_k(n)$
        \item long, if it is $k$-long for all $k$.
    \end{itemize}
\end{defi}

It is clear, that $k$-long families are also $l$-long for $l<k$.

\begin{prop}
    Long families of polytopes exist.
\end{prop}
\begin{proof}
    Let us construct $B^n(k)$, an $n$-dimensional polytope that has a chain of $k$ facets. Fix $k$ points along the meridian of $n$-sphere, draw an $n$-simplex on the surface of the sphere around every point, oriented along the meridian, and take the convex hull. The family $P^n = B^n(n)$ is long. 
\end{proof}

Note that by fixing $k$ and setting $P^n = B^n(k)$ we obtain a family, that is $k$-long, but not $(k+1)$-long.

\begin{figure}[H]
    \centering
    \includegraphics[width=0.5\linewidth]{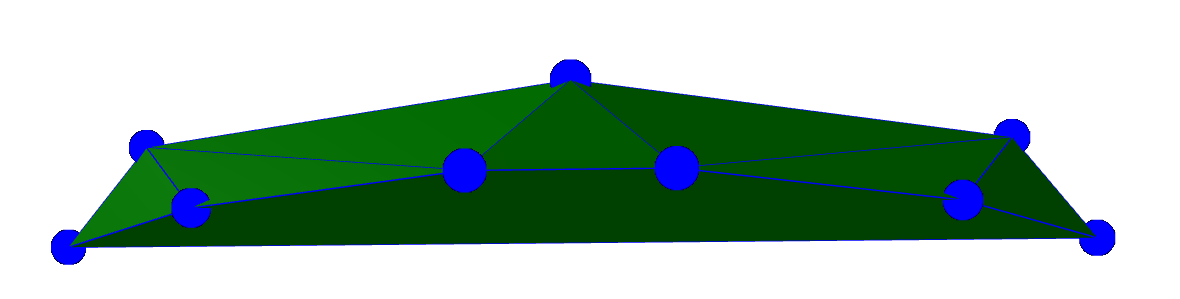}
    \caption{The polytope $B^3(3)$}
    \label{fig:long}
\end{figure}

While long families exist, the definition is a bit too restrictive for our purposes. We therefore introduce the asymptotic version.

\begin{defi}
    A family of polytopes $P^\bullet$
    \begin{itemize}
        \item has asymptotic $k$-length $\alpha_k$, if $$\lim_{n \to \infty} \frac{e_k(P^n)}{E_k(n)} = \alpha_k$$
        \item has asymptotic total length $\alpha$, if asymptotic $k$-lengths $\alpha_k$ are well-defined for every $k$ and 
        $$ \lim_{k \to \infty} \alpha_k = \alpha $$
    \end{itemize}
\end{defi}

Asymptotic $k$-lengths, when they exist, can vary anywhere between 0 and 1. 

We introduce some auxiliary technical notation. For a polytope $P$, let $f_k(P)$ denote the sum of face dimensions in a face $k$-chain of minimal excess in $P$, and set 
$$ \beta_k(P^\bullet) = \lim_{n \to \infty} \frac{f_k(P^n)}{n}$$

Then we have 

$$ \alpha_k(P^\bullet) = \lim_{n \to \infty} \frac{(n-1) - (f_k(P^n) - k)}{(1-k)n+2k-1} = \lim_{n \to \infty} \frac{n - f_k(P^n)}{(1-k)n} = \frac{\beta_k(P^\bullet)-1}{k-1}$$
or, vice versa,
$$\beta_k (P^\bullet) = 1+(k-1)\alpha_k(P^\bullet).$$

 This $\beta$-notation will be convenient for some computations. We now make a simple observation about how asymptotic $k$-lengths of a family are connected with each other for different $k$.

\begin{prop}
\label{upperboundary}
A family with asymptotic $k$-length $\alpha_k$ has asymptotic $l$-length at most $\frac{l(k-1)}{k(l-1)}\alpha_k + \frac{l-k}{k(l-1)} $ for $l \geq k$ and therefore asymptotic total length at most $\frac{k-1}{k} \alpha_k + \frac{1}{k}$.
\end{prop}
\begin{proof}
    We use $\beta_k$-notation. For $l \geq k$ we have 
    $$f_l(P^n) \leq \frac{l}{k} f_k(P^n),$$ 
    otherwise there can be chosen a subchain within the face $l$-chain of minimal excess in $P^n$, whose sum of face dimensions is larger than $f_k(P^n)$. Therefore
    $$\beta_l \leq \frac{l}{k} \beta_k.$$
    We then have 
    $$ \alpha_l = \frac{\beta_l-1}{l-1} \leq \frac{\frac{l}{k}\beta_k-1}{l-1} = \frac{l (1+(k-1)\alpha_k) - k}{k(l-1)} = \frac{l(k-1)}{k(l-1)}\alpha_k + \frac{l-k}{k(l-1)} $$
    Taking the limit of the latter expression for $l \to \infty$ we obtain the upper bound for the total length.
\end{proof}

\section{Permutahedra}
A permutahedron $\mathbb{P}^n$ is a convex hull of points  $(\sigma(\mathbf{e}_1), \ldots, \sigma(\mathbf{e}_n))$ for all the permutations $\sigma \in S_n$, where $\mathbf{e}_i$ are a basis of $\mathbb{R}^n$. It is an $(n-1)$-dimensional polytope embedded in $\mathbb{R}^n$. Taking the linear functional $$\psi(\mathbf{v}) = \sum_{i=1}^n i \mathbf{v}_i $$
we obtain the weak Bruhat order on the vertices the permutahedron. Faces of the permutahedron can be labeled with ordered partitions of the set $\{1, \ldots, n \}$, with face inclusion corresponding to ordered partition refinement. If the ordered partition is into $n$ parts -- i.e. a permutation -- then the coordinates of the corresponding vertex are given by the inverse permutation.

For an ordered partition $\tau$ and an index $i$, let $\tau(i)$ denote the number of part in $\tau$ that contains $i$, e.g. for $\tau = [1,3][2,4]$ we have $\tau(4)=2$. Let $\tau$ and $\rho$ be two ordered partitions. Then $\tau \leq \rho$ with respect to weak Bruhat order, if for every pair $i < j$ we have either $\tau(i)<\tau(j)$ or $\rho(i)>\rho(j)$ (the inequalities are strict). 

The main result of the section is the following.

\begin{theo}
\label{theoremperm}
    The family of permutahedra has asymptotic $k$-length 1 for all $k$ and therefore asymptotic total length 1.
\end{theo}

To prove the theorem, we introduce certain special face chains, named {\em zebra chains}.

\subsection{Zebra chains}
Arrange numbers from $1$ to $mn$ in a $m \times n$ rectangle left to right, line by line. The zebra chain $\mathbf{Zeb}^k(m,n)$ will be a chain of $k$ faces in $\mathbb{P}^{mn}$ (note that $k$ only affects where we stop).

\begin{defi}
A single part of $ \mathbf{Zeb}^k(m,n)_1$ consists of all points that lie on some horizontal line. Parts are ordered top to bottom.

For $1<l \leq k$, a single part of $ \mathbf{Zeb}^k(m,n)_l$ consists of all points that lie on some line with slope $1/(l-2)$ (infinite slope means vertical). Parts are ordered left to right. See Figure \ref{fig:zebra} for example.
\end{defi}

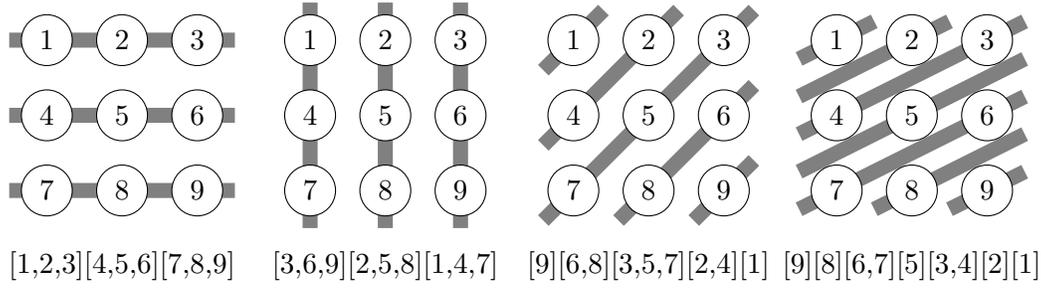
\begin{figure}
    \centering
    \begin{tikzpicture}    	
    	\draw [line width=2mm, color = gray] (0-0.5,0) -- (2+0.5,0);
        \draw [line width=2mm, color = gray] (0-0.5,1) -- (2+0.5,1);
        \draw [line width=2mm, color = gray] (0-0.5,2) -- (2+0.5,2);

        \node[circle, draw=black, fill = white] at (0,2){1};
        \node[circle, draw=black, fill = white] at (1,2){2};
        \node[circle, draw=black, fill = white] at (2,2){3};

        \node[circle, draw=black, fill = white] at (0,1){4};
        \node[circle, draw=black, fill = white] at (1,1){5};
        \node[circle, draw=black, fill = white] at (2,1){6};

        \node[circle, draw=black, fill = white] at (0,0){7};
        \node[circle, draw=black, fill = white] at (1,0){8};
        \node[circle, draw=black, fill = white] at (2,0){9};

        \node at (1,-1){[1,2,3][4,5,6][7,8,9]};

    	\draw [line width=2mm, color = gray] (0+3.5,0-0.5) -- (0+3.5,2+0.5);
        \draw [line width=2mm, color = gray] (1+3.5,0-0.5) -- (1+3.5,2+0.5);
        \draw [line width=2mm, color = gray] (2+3.5,0-0.5) -- (2+3.5,2+0.5);

        \node[circle, draw=black, fill = white] at (0+3.5,2){1};
        \node[circle, draw=black, fill = white] at (1+3.5,2){2};
        \node[circle, draw=black, fill = white] at (2+3.5,2){3};

        \node[circle, draw=black, fill = white] at (0+3.5,1){4};
        \node[circle, draw=black, fill = white] at (1+3.5,1){5};
        \node[circle, draw=black, fill = white] at (2+3.5,1){6};

        \node[circle, draw=black, fill = white] at (0+3.5,0){7};
        \node[circle, draw=black, fill = white] at (1+3.5,0){8};
        \node[circle, draw=black, fill = white] at (2+3.5,0){9};

       \node at (1+3.5,-1){[3,6,9][2,5,8][1,4,7]};

        \draw [line width=2mm, color = gray] (0+3.5*2-0.4,2-0.4) -- (0+3.5*2+0.4,2+0.4);
        \draw [line width=2mm, color = gray] (0+3.5*2-0.4,1-0.4) -- (1+3.5*2+0.4,2+0.4);
        \draw [line width=2mm, color = gray] (0+3.5*2-0.4,0-0.4) -- (2+3.5*2+0.4,2+0.4);
        \draw [line width=2mm, color = gray] (1+3.5*2-0.4,0-0.4) -- (2+3.5*2+0.4,1+0.4);
        \draw [line width=2mm, color = gray] (2+3.5*2-0.4,0-0.4) -- (2+3.5*2+0.4,0+0.4);

        \node[circle, draw=black, fill = white] at (0+3.5*2,2){1};
        \node[circle, draw=black, fill = white] at (1+3.5*2,2){2};
        \node[circle, draw=black, fill = white] at (2+3.5*2,2){3};

        \node[circle, draw=black, fill = white] at (0+3.5*2,1){4};
        \node[circle, draw=black, fill = white] at (1+3.5*2,1){5};
        \node[circle, draw=black, fill = white] at (2+3.5*2,1){6};

        \node[circle, draw=black, fill = white] at (0+3.5*2,0){7};
        \node[circle, draw=black, fill = white] at (1+3.5*2,0){8};
        \node[circle, draw=black, fill = white] at (2+3.5*2,0){9};

       \node at (1+2*3.5,-1){[9][6,8][3,5,7][2,4][1]};          

        \draw [line width=2mm, color = gray] (0+3.5*3-0.5,2-0.25) -- (0+3.5*3+0.5,2+0.25);
        \draw [line width=2mm, color = gray] (1+3.5*3-1.5,2-0.75) -- (1+3.5*3+0.5,2+0.25);
        \draw [line width=2mm, color = gray] (0+3.5*3-0.5,1-0.25) -- (2+3.5*3+0.5,2+0.25);
        \draw [line width=2mm, color = gray] (1+3.5*3-1.5,1-0.75) -- (1+3.5*3+1.5,1+0.75);
        \draw [line width=2mm, color = gray] (0+3.5*3-0.5,0-0.25) -- (2+3.5*3+0.5,1+0.25);
        \draw [line width=2mm, color = gray] (1+3.5*3-0.5,0-0.25) -- (1+3.5*3+1.5,0+0.75);
        \draw [line width=2mm, color = gray] (2+3.5*3-0.5,0-0.25) -- (2+3.5*3+0.5,0+0.25);

        \node[circle, draw=black, fill = white] at (0+3.5*3,2){1};
        \node[circle, draw=black, fill = white] at (1+3.5*3,2){2};
        \node[circle, draw=black, fill = white] at (2+3.5*3,2){3};

        \node[circle, draw=black, fill = white] at (0+3.5*3,1){4};
        \node[circle, draw=black, fill = white] at (1+3.5*3,1){5};
        \node[circle, draw=black, fill = white] at (2+3.5*3,1){6};

        \node[circle, draw=black, fill = white] at (0+3.5*3,0){7};
        \node[circle, draw=black, fill = white] at (1+3.5*3,0){8};
        \node[circle, draw=black, fill = white] at (2+3.5*3,0){9};

        \node at (1+3*3.5,-1){[9][8][6,7][5][3,4][2][1]};
    \end{tikzpicture}
    \caption{The zebra chain $ \mathbf{Zeb}^4(3,3)$}
    \label{fig:zebra}
\end{figure}

\begin{prop}
The zebra chain $ \mathbf{Zeb}^k(m,n)$ is indeed a face chain, i.e. $$ \mathbf{Zeb}^k(m,n)_l \leq  \mathbf{Zeb}^k(m,n)_{l+1}$$ in the weak order. 
\end{prop}
\begin{proof}
    We first check that $ \mathbf{Zeb}^k(m,n)_1 \leq  \mathbf{Zeb}^k(m,n)_{2}$. Take a pair $i<j$. If $i$ and $j$ lie on different lines in the rectangle, then $j$ is lower than $i$ and they are in order in $\mathbf{Zeb}^k(m,n)_1$. If $i$ and $j$ lie on the same line, then $i$ is to the left of $j$ and they are in disorder  $\mathbf{Zeb}^k(m,n)_2$.

    We now check the inequalities for the rest of zebra chain. For a pair $i<j$, let $\alpha$ be the slope of the line connecting $i$ and $j$. If $\alpha < 0$, then the pair is in disorder in $\mathbf{Zeb}(m,n)_l$ for all $l\geq 2$. If $\alpha \geq 0$, then $i$ and $j$ are in order in $\mathbf{Zeb}^k(m,n)_l$ when $\alpha < 1/(l-2)$, neither in order nor in disorder when $\alpha = 1/(l-2)$ and in disorder when $\alpha > 1/(l-2)$. We therefore observe the change of state from ``in order" to ``neither" to ``in disorder" with the growth of $l$.
\end{proof}

\subsection{Length computations}
The key idea of this section is that for squares, zebra chains have exceptionally negative excesses.

\begin{prop}
\label{zebradim}
For $l = 1$ or $l = 2$, we have 
$$ \dim \mathbf{Zeb}^k(n,n)_l = n^2-n$$
For $l > 2$ and $n \geq l-2 $, we have
    $$ \dim \mathbf{Zeb}^k(n,n)_l = n^2-(l-1)n+1$$
        
\end{prop}

\begin{proof}
    The dimension of a face in $\mathbb{P}^{n^2}$ is the $n^2$ minus the number of parts in the corresponding ordered partition. The first part of the proposition is obvious. One must then verify for any $l>2$ that the number of parts in $\mathbf{Zeb}^k(n,n)_l$ is $(l-1)n-1$. We fix $l$ and prove the proposition by induction on $n$. For $n = l-2$ all the points in the diagram for  $\mathbf{Zeb}^k(n,n)_l$ are in different parts, providing the base of induction. Now assume the proposition for $n$. In the diagram for $\mathbf{Zeb}^k(n+1,n+1)_l$, there are $l-1$ new parts (see Figure \ref{fig:inductivestep}).
\end{proof}

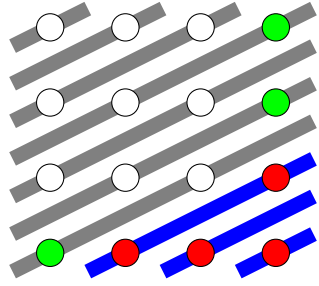
\begin{figure}
    \centering
    \begin{tikzpicture}
        \draw [line width=2mm, color = gray] (0+3.5*3-0.5,2-0.25) -- (0+3.5*3+0.5,2+0.25);
        \draw [line width=2mm, color = gray] (1+3.5*3-1.5,2-0.75) -- (1+3.5*3+0.5,2+0.25);
        \draw [line width=2mm, color = gray] 
        
        (0+3.5*3-0.5,1-0.25) -- (2+3.5*3+0.5,2+0.25);
        \draw [line width=2mm, color = blue] 
        (1+3.5*3-0.5,-1-0.25) -- (3+3.5*3+0.5,0+0.25);
        \draw [line width=2mm, color = blue] 
        (2+3.5*3-0.5,-1-0.25) -- (2+3.5*3+1.5,-1+0.75);
        \draw [line width=2mm, color = blue] 
        (3+3.5*3-0.5,-1-0.25) -- (3+3.5*3+0.5,-1+0.25);
        
        \draw [line width=2mm, color = gray] 
        (1+3.5*3-1.5,1-0.75) -- (1+3.5*3+2.5,1+1.25);
        \draw [line width=2mm, color = gray] (0+3.5*3-0.5,0-0.25) -- (2+3.5*3+1.5,1+0.75);
        \draw [line width=2mm, color = gray] (1+3.5*3-1.5,0-0.75) -- (1+3.5*3+2.5,0+1.25);
        \draw [line width=2mm, color = gray] (2+3.5*3-2.5,0-1.25) -- (2+3.5*3+1.5,0+0.75);

        \node[circle, draw=black, fill = white] at (0+3.5*3,2){};
        \node[circle, draw=black, fill = white] at (1+3.5*3,2){};
        \node[circle, draw=black, fill = white] at (2+3.5*3,2){};
        \node[circle, draw=black, fill = green] at (3+3.5*3,2){};

        \node[circle, draw=black, fill = white] at (0+3.5*3,1){};
        \node[circle, draw=black, fill = white] at (1+3.5*3,1){};
        \node[circle, draw=black, fill = white] at (2+3.5*3,1){};
        \node[circle, draw=black, fill = green] at (3+3.5*3,1){};

        \node[circle, draw=black, fill = white] at (0+3.5*3,0){};
        \node[circle, draw=black, fill = white] at (1+3.5*3,0){};
        \node[circle, draw=black, fill = white] at (2+3.5*3,0){};
        \node[circle, draw=black, fill = red] at (3+3.5*3,0){};

        \node[circle, draw=black, fill = green] at (0+3.5*3,-1){};
        \node[circle, draw=black, fill = red] at (1+3.5*3,-1){};
        \node[circle, draw=black, fill = red] at (2+3.5*3,-1){};
        \node[circle, draw=black, fill = red] at (3+3.5*3,-1){};
        
    \end{tikzpicture}
    \caption{The inductive step in the proof of Prop. \ref{zebradim}: green points join existing (gray) groups and red points form the $l-1$ new (blue) groups}
    \label{fig:inductivestep}
\end{figure}

Take $n$ sufficiently large so that Proposition \ref{zebradim} is applicable. We compute the excess:

$$ \mathrm{exc}(\mathbf{Zeb}^k(n,n)) = (n^2-2) - \sum_{i = 1} ^k (\dim \mathbf{Zeb}^k(n,n)_i - 1) = $$
$$ =  (n^2-2) - 2(n^2-n-1) - \sum_{l = 3} ^k (n^2 - (l-1)n) = (1-k) n^2 + \left (1 + \frac{(k-1)k}{2}\right )n + 1.$$

The following is now straightforward, by looking at the coefficient near $n^2$.

\begin{prop}
    $$\lim_{n \to \infty} \frac{\mathrm{exc}(\mathbf{Zeb}^k(n,n))} {E_k(n^2) } = 1. $$
\end{prop}

To compute asymptotic lengths, it suffices to work with zebra chains in the squares, because for $a \in [n^2,(n+1)^2]$, one can consider a partial version of the zebra chain, denoted $\mathbf{F}^k(a)$: again arrange numbers from $1$ to $a$ in a $(n+1) \times (n+1)$ square left to right, line by line, and for $l$th term of the chain, form the groups just like for a zebra chain (see Figure \ref{fig:zebra_partial}). We then have 

$$ \dim \mathbf{Zeb}^k(n,n)_l \leq \dim \mathbf{F}^k(a)_l \leq \dim \mathbf{Zeb}^k(n+1,n+1)_l$$

and therefore

$$     \frac{\mathrm{exc}(\mathbf{Zeb}^k(n,n))} {E_k(n^2) } \leq \frac{\mathrm{exc}(\mathbf{F}^k(a))}{E_k(a)} \leq \frac{\mathrm{exc}(\mathbf{Zeb}^k(n+1,n+1))} {E_k((n+1)^2) } $$

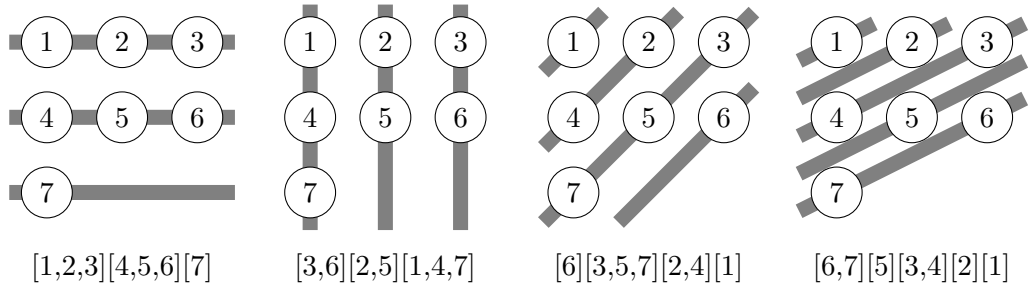
\begin{figure}
    \centering
    \begin{tikzpicture}    	
    	\draw [line width=2mm, color = gray] (0-0.5,0) -- (2+0.5,0);
        \draw [line width=2mm, color = gray] (0-0.5,1) -- (2+0.5,1);
        \draw [line width=2mm, color = gray] (0-0.5,2) -- (2+0.5,2);

        \node[circle, draw=black, fill = white] at (0,2){1};
        \node[circle, draw=black, fill = white] at (1,2){2};
        \node[circle, draw=black, fill = white] at (2,2){3};

        \node[circle, draw=black, fill = white] at (0,1){4};
        \node[circle, draw=black, fill = white] at (1,1){5};
        \node[circle, draw=black, fill = white] at (2,1){6};

        \node[circle, draw=black, fill = white] at (0,0){7};

        \node at (1,-1){[1,2,3][4,5,6][7]};

    	\draw [line width=2mm, color = gray] (0+3.5,0-0.5) -- (0+3.5,2+0.5);
        \draw [line width=2mm, color = gray] (1+3.5,0-0.5) -- (1+3.5,2+0.5);
        \draw [line width=2mm, color = gray] (2+3.5,0-0.5) -- (2+3.5,2+0.5);

        \node[circle, draw=black, fill = white] at (0+3.5,2){1};
        \node[circle, draw=black, fill = white] at (1+3.5,2){2};
        \node[circle, draw=black, fill = white] at (2+3.5,2){3};

        \node[circle, draw=black, fill = white] at (0+3.5,1){4};
        \node[circle, draw=black, fill = white] at (1+3.5,1){5};
        \node[circle, draw=black, fill = white] at (2+3.5,1){6};

        \node[circle, draw=black, fill = white] at (0+3.5,0){7};

       \node at (1+3.5,-1){[3,6][2,5][1,4,7]};

        \draw [line width=2mm, color = gray] (0+3.5*2-0.4,2-0.4) -- (0+3.5*2+0.4,2+0.4);
        \draw [line width=2mm, color = gray] (0+3.5*2-0.4,1-0.4) -- (1+3.5*2+0.4,2+0.4);
        \draw [line width=2mm, color = gray] (0+3.5*2-0.4,0-0.4) -- (2+3.5*2+0.4,2+0.4);
        \draw [line width=2mm, color = gray] (1+3.5*2-0.4,0-0.4) -- (2+3.5*2+0.4,1+0.4);

        \node[circle, draw=black, fill = white] at (0+3.5*2,2){1};
        \node[circle, draw=black, fill = white] at (1+3.5*2,2){2};
        \node[circle, draw=black, fill = white] at (2+3.5*2,2){3};

        \node[circle, draw=black, fill = white] at (0+3.5*2,1){4};
        \node[circle, draw=black, fill = white] at (1+3.5*2,1){5};
        \node[circle, draw=black, fill = white] at (2+3.5*2,1){6};

        \node[circle, draw=black, fill = white] at (0+3.5*2,0){7};

       \node at (1+2*3.5,-1){[6][3,5,7][2,4][1]};          

        \draw [line width=2mm, color = gray] (0+3.5*3-0.5,2-0.25) -- (0+3.5*3+0.5,2+0.25);
        \draw [line width=2mm, color = gray] (1+3.5*3-1.5,2-0.75) -- (1+3.5*3+0.5,2+0.25);
        \draw [line width=2mm, color = gray] (0+3.5*3-0.5,1-0.25) -- (2+3.5*3+0.5,2+0.25);
        \draw [line width=2mm, color = gray] (1+3.5*3-1.5,1-0.75) -- (1+3.5*3+1.5,1+0.75);
        \draw [line width=2mm, color = gray] (0+3.5*3-0.5,0-0.25) -- (2+3.5*3+0.5,1+0.25);

        \node[circle, draw=black, fill = white] at (0+3.5*3,2){1};
        \node[circle, draw=black, fill = white] at (1+3.5*3,2){2};
        \node[circle, draw=black, fill = white] at (2+3.5*3,2){3};

        \node[circle, draw=black, fill = white] at (0+3.5*3,1){4};
        \node[circle, draw=black, fill = white] at (1+3.5*3,1){5};
        \node[circle, draw=black, fill = white] at (2+3.5*3,1){6};

        \node[circle, draw=black, fill = white] at (0+3.5*3,0){7};

        \node at (1+3*3.5,-1){[6,7][5][3,4][2][1]};
    \end{tikzpicture}
    \caption{Partial zebra chain $\mathbf{F}^4(7)$}
\label{fig:zebra_partial}
\end{figure}

We have therefore proved Theorem \ref{theoremperm}.

\section{Associahedra}

An associahedron $K(n)$, also known as Stasheff polytope after \cite{Sta}, is any polytope whose face poset is isomorphic to that of planar rooted trees on $n$ leaves (alternatively: of bracketings of $n$ letters), ordered by inner edge contraction (alternatively: by bracket erasion). A well-known realization of associahedra is due to Loday \cite{Lod}. Loday's associahedra are deformed permutahedra in the sense of \cite{Pos} and moreover removahedra (i.e. can be obtained from the permutahedron by removing some facets). The same linear functional as in the previous section induces the {\em Tamari order} on vertices of associahedra, which correspond to {\em binary} planar rooted trees (alternatively: to binary bracketings). Generating relations for this order are given by right edge rotations (alternatively: by applications of the associative law), see Figure \ref{fig:tamari}. 

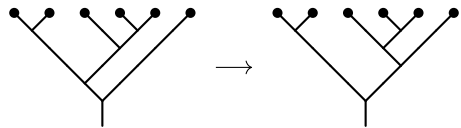
\begin{figure}[H]
    \centering
\begin{tikzpicture}

\node[scale = 2] (A) at (0,0) {\RS{ I  {L  {LLL {l} {r} } {RR {Ll} {R {l} {r} } } } {RRRRr} }};

\node[scale = 2] (B) at (3.5,0) {\RS{ I {LLLL {l} {r}} {RR {L {Ll} {R {l} {r} }} RRr}}};

\draw[->] (A) -- (B);

\end{tikzpicture}
    \caption{A generating relation for Tamari order on the vertices of $K(6)$}
    \label{fig:tamari}
\end{figure}

Associahedral cellular diagonals are described by {\em Loday's magic formula} -- see \cite{SU22} for comparison of the existing formulas. It is a manifestation of the fact that associahedra are $2$-short. For readers' convenience we provide the proof of this fact.

\begin{prop}
    Any length $2$ chain in an associahedron $K(n)$ has positive excess.
\end{prop}

\begin{proof}
    We need to see that any pair of faces $F_1 \leq F_2$ has $\dim F_1 + \dim F_2 \leq n-2$. For a vertex $V$, set $D_{\downarrow}(V)$ to be the maximal dimension of a face having $V$ as its top vertex, and set $D_{\uparrow}(V)$ to be the maximal dimension of a face having $V$ as its bottom vertex. Clearly $D_{\downarrow}(V)$ is equal to the number of right-leaning inner edges in the corresponding binary tree and $D_{\uparrow}(V)$ is equal to the number of left-leaning inner edges. We then need to verify that for any pair $V \leq W$, we have $ D_{\downarrow}(V) + D_{\uparrow}(W) \leq n-2$. This is clear since $ D_{\downarrow}(V) + D_{\uparrow}(V) = n-2$ and a tree rotation does not increase the number of left-leaning edges.
\end{proof}

Proposition \ref{upperboundary} then implies the following.

\begin{cor}
    The family of associahedra has $k$-length at most $\frac{k-2}{2k-2}$ for $k>2$, and total length at most $\frac{1}{2}$.
\end{cor}

These upper bounds happen to be strict, as we now prove, again by introducing certain special chains, named {\em thuja chains}.

\subsection{Thuja chains}
Let an $n$-thuja $Th(n)$ be the planar tree on $n$ leaves corresponding to a non-binary bracketing with $\lfloor \frac{n}{2} \rfloor$ pairs of brackets, where first all the brackets open and then all the brackets close -- such as $(1(2(34)5)6)$ or $(1(2(345)6)7)$ (see Figure \ref{fig:thuja}). We have $\dim Th(n) = \lfloor \frac {n-1} {2} \rfloor $. Let $T_{\max} (m)$ be the planar binary tree on $m$ leaves maximal in the Tamari order.

\begin{figure}
    \centering
\begin{tikzpicture}
\node[scale = 2] (A) at (3.5,0) {\RS{ I {LLLl}  {I {LLl} {I {Ll} {Ii} {Rr}} {RRr} } {RRRr} }};

\node[scale = 2] (A) at (0,0) {\RS{ I {LLl}  {I {Ll} {I {l} {r}} {Rr} } {RRr} }};
\end{tikzpicture}
    \caption{Trees $Th(6)$ and $Th(7)$}
    \label{fig:thuja}
\end{figure}
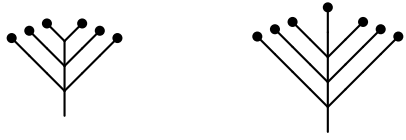

\begin{defi}
A thuja chain $\mathbf{Th}(n)$ is a chain of $n-1$ faces in $K(n)$ where $\mathbf{Th}(n)_l$ is obtained by grafting the $(n-l+1)$-thuja $Th(n-l+1)$ into the rightmost leaf of  $T_{\max} (l)$ (see Figure \ref{fig:thujachain}).
\end{defi}

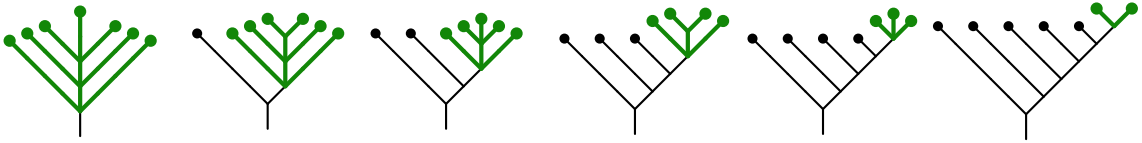
\begin{figure}
    \centering
\begin{tikzpicture}
\node[scale = 2] (A) at (0,0) {\RS{ I {AAAa}  {B {AAa} {B {Aa} {Bb} {Cc}} {CCc} } {CCCc} }};

\node[scale = 2] (A) at (2.5,0) {\RS{ I {LLLl} {R {AAa} {B {Aa} {B {a} {c}} {Cc}} {CCc}} }};

\node[scale = 2] (A) at (5,0) {\RS{I {LLLl} {R {LLl} {R {Aa}{B {a} {b} {c}}{Cc}}}} };

\node[scale = 2] (A) at (7.5,0) {\RS{I {LLLl} {R {LLl} {R {Ll } {R {Aa} {B {a} {c} } {Cc} }}}}};

\node[scale = 2] (A) at (10,0) {\RS{I {LLLl} {R {LLl} {R {Ll } {R {l} {R {a} {b} {c} } }}}}};

\node[scale = 2] (A) at (12.7,0) {\RS{I {LLLLl} {R {LLLl} {R {LLl } {R {Ll} {R {l} {R {a} {c} }}}}}}};

\end{tikzpicture}
    \caption{The thuja chain $\mathbf{Th}(7)$, with grafted $(n-l+1)$-thujas marked green}
    \label{fig:thujachain}
\end{figure}

\begin{prop}
    The thuja chain $ \mathbf{Th}(n)$ is indeed a face chain, i.e. $$ \mathbf{Th}(n)_l \leq  \mathbf{Th}(n)_{l+1}$$ in the Tamari order. 
\end{prop}

\begin{proof}
    If suffices to check $ \mathbf{Th}(i)_1 \leq  \mathbf{Th}(i)_{1}$ for all $i$, since later steps in the thuja chain are obtained by grafting this one into appropriate $T_{\max} (m)$. We have $ \max \mathbf{Th}(i)_1 =  \min \mathbf{Th}(i)_{1}$, proving the claim.
\end{proof}

Since $T_{\max} (l)$ is 0-dimensional and grafting is dimension-additive, we have $\dim \mathbf{Th}(n)_l = \dim Th(n-l+1) = \lfloor \frac {n-l} {2} \rfloor $.

\subsection{Length computations}

Take for $n>k$ and let $\mathbf{Th}^k(n)$ denote the thuja chain taken up to $k$th term. We compute its excess:
$$ \mathrm{exc}(\mathbf{Th}^k(n)) = (n-3) - \sum_{l = 1} ^k (\dim \mathbf{Th}(n)_l - 1) = n-3- \sum_{l = 1}^k \left ( \left \lfloor \frac {n-l} {2} \right \rfloor -1 \right )  $$

We then have 

   $$\lim_{n \to \infty} \frac{\mathrm{exc}(\mathbf{Th}^k(n))} {E_k(n) } = \frac{1-\frac{k}{2}}{1-k} = \frac{k-2}{2k-2}.$$

   The limit at $k \to \infty$ is $\frac{1}{2}$. We have therefore proved the following. 

\begin{theo}
 The family of associahedra is $2$-short, has asymptotic $k$-length $\frac{k-2}{2k-2}$ for $k>2$, and asymptotic total length  $\frac{1}{2}$.
   
\end{theo}
    
{\huge Happy Birthday, Jim!}

\end{document}